\documentclass[12pt]{amsart}
\usepackage{latexsym, amssymb, amsmath}

\newtheorem{theorem}{Theorem}[section]
\newtheorem{lemma}[theorem]{Lemma}
\newtheorem{corollary}[theorem]{Corollary}
\newtheorem{proposition}[theorem]{Proposition}
\newtheorem{remark}[theorem]{Remark}
\theoremstyle{definition}
\newtheorem{definition}[theorem]{Definition}

\title[4-dim pseudo-umbilical biharmonic submanifolds]
{Four-dimensional pseudo-umbilical biharmonic submanifolds in six-dimensional Riemannian manifolds of constant sectional curvature}

\author{Shun Maeta}


\address{\footnotesize{
Department of Mathematics, Chiba University, 1-33, Yayoicho, Inage, Chiba, 263-8522, Japan.
}
 }
\footnotesize{
\email{shun.maeta@faculty.gs.chiba-u.jp}
}

\begin{document} 
\begin{abstract} 
In 1989, Dimitri\'c showed that any pseudo-umbilical biharmonic submanifold in a Euclidean space is minimal when the dimension of the submanifold is different from four. However, the four-dimensional case has remained open for nearly four decades.
This difficulty comes from the fact that the tangential part of the biharmonic equation vanishes identically in dimension four. 
Therefore, we consider biminimal submanifolds, whose defining equation generalizes the normal part of the biharmonic equation.
In this paper, we show that any connected pseudo-umbilical 0-biminimal submanifold of codimension two and dimension at least three in a Riemannian manifold of constant sectional curvature $c$ has constant mean curvature. Furthermore, when $c$ is non-positive, we prove that it is minimal. As a corollary, we show that any four-dimensional pseudo-umbilical biharmonic submanifold in a six-dimensional Riemannian manifold of non-positive constant sectional curvature is minimal. These results also provide partial affirmative answers to Chen's conjecture, generalized Chen's conjecture for non-positive constant sectional curvature, and the Balmu\c{s}-Montaldo-Oniciuc conjecture.
\end{abstract}
\maketitle 
\markboth{4-dim pseudo-umbilical biharmonic submanifolds} 
{Shun Maeta}


\section{Introduction}\label{intro}

Biharmonic submanifolds in a Euclidean space $\mathbb{E}^n$ are submanifolds with harmonic mean curvature vector $\Delta {\bf H}=0$ in $\mathbb{E}^n$ \cite{Chen1988}, where $\Delta$ and ${\bf H}$ denote the (non-positive) Laplace operator 
and the mean curvature vector field of the submanifold $M^m$, respectively. 
Therefore biharmonic submanifolds generalize minimal submanifolds.
Interestingly, biharmonic submanifolds can also be studied through the variational theory of biharmonic maps \cite{EL1983} (see also \cite{Jiang1986}).
The theory of biharmonic submanifolds has developed rapidly, in particular, from the perspective of nonexistence theory (see the survey \cite{Chen2025}). 
In 1988, Chen posed the following conjecture ``Any biharmonic submanifold in $\mathbb{E}^n$ is minimal". This is known as Chen's conjecture \cite{Chen1988}. 

There are many partial affirmative answers to this conjecture for hypersurfaces.
Chen and Jiang independently proved that every biharmonic surface in $\mathbb{E}^3$ is minimal \cite{Chen1988,Jiang1987}. 
Hasanis and Vlachos \cite{HV1995} and Defever \cite{Defever1998} proved that every biharmonic hypersurface in $\mathbb{E}^4$ is minimal. 
Recently, Fu, Hong, and Zhan proved Chen's conjecture for biharmonic hypersurfaces in $\mathbb{E}^5$ \cite{FHZ2021} and in $\mathbb{E}^6$ \cite{FHZ2023}. 

However, much less is known in higher codimension. Akutagawa and the author showed that any properly immersed biharmonic submanifold in $\mathbb{E}^n$ is minimal (cf. \cite{AM2013}, see also \cite{Maeta2014}). 
Under the assumption that the normalized mean curvature vector field is parallel, any biharmonic surface in $\mathbb{E}^n$ is minimal (cf. \cite{Chen2019}, see also \cite{MU2013}, \cite{ST2018}). Recently, the author also showed that any simple rotational biharmonic surface in $\mathbb{E}^4$ is minimal \cite{Maeta2026}.

Chen's conjecture was also studied for geometrically natural submanifolds by Dimitri\'c. He showed that (1) any biharmonic curve is an open part of a straight line, (2) any biharmonic submanifold with constant mean curvature is minimal, and (3) any totally umbilical biharmonic submanifold is minimal \cite{Dimitric1989}. 
Therefore, the next natural class is pseudo-umbilical submanifolds, which generalize totally umbilical submanifolds by requiring umbilicity only in the direction of the mean curvature vector. 
This condition singles out a normal direction determined by the immersion itself, while allowing the second fundamental form to vary in the remaining normal directions.
Pseudo-umbilical submanifolds and related classes of submanifolds have been studied extensively (cf. \cite{Chen1971b}, \cite{Chen1976}, \cite{CL1972}, \cite{Fetcu2012}, \cite{Smyth1973}, \cite{YI1971}, and \cite{Yau1974}).
In 1989, Dimitri\'c showed that every pseudo-umbilical biharmonic submanifold $M^m$ in $\mathbb{E}^n$ is minimal when $m\neq4$ \cite{Dimitric1989}.
However, the four-dimensional case has remained open for nearly four decades.
Analogous results hold in general space forms. 
For $m\neq4$, Balmu\c{s}, Montaldo, and Oniciuc proved in 2008 that pseudo-umbilical biharmonic submanifolds in spheres have constant mean curvature (cf. \cite{BMO2008}, see also \cite{LO2016}), while Caddeo, Montaldo, and Oniciuc proved in 2002 that pseudo-umbilical biharmonic submanifolds in hyperbolic spaces are minimal (cf. \cite{CMO2002}).

The exceptional role of dimension four can be seen directly from the biharmonic equation.
Let $\{e_1,\cdots,e_m\}$ be a local orthonormal frame and set $f=|{\bf H}|$. 
It is well known that the necessary and sufficient condition for a submanifold $M^m$ in a Riemannian manifold $N^n(c)$ of constant sectional curvature $c$ to be biharmonic can be decomposed into normal and tangential parts (cf.~\cite{Chen-Ishikawa-1}, \cite{Chen-Ishikawa-2}, \cite{CMO2002}):
\[
\begin{cases} 
\ \ \Delta^{\perp} {\bf H} - \sum_{i=1}^m B(A_{\bf H}e_i, e_i) +mc{\bf H}= 0, \\ 
\ \ m\,\nabla |{\bf H}|^2 + 4\sum_{i=1}^mA_{\nabla^{\perp}_{e_i} {\bf H}}e_i =0 , \\  
\end{cases} 
\]
where $\Delta^{\perp}$ is the (non-positive) Laplace operator associated with the normal connection $\nabla^{\perp}$, $B$ is the second fundamental form of $M$ in $N^n(c)$, $A_\xi$ is the Weingarten map with respect to $\xi$, and $\nabla |{\bf H}|^2$ is the gradient of $|{\bf H}|^2$.
If $M$ is pseudo-umbilical, then 
\[
A_{\bf H}=f^2 {\rm Id}.
\]
Hence a straightforward computation using the Codazzi equation gives 
\[
\sum_{i=1}^mA_{\nabla^\perp_{e_i}{\bf H}}e_i=\frac{2-m}{2} \nabla f^2
\]
(cf. \cite{BMO2013}).
Consequently, the tangential biharmonic equation becomes 
\[
(m-4)\nabla f^2=0.
\]
Therefore, for $m\neq4$, $f$ is constant. 
The situation changes completely when $m=4$. 
In this case, the tangential equation vanishes identically and hence  imposes no restriction on $\nabla f$. 
Therefore, the rigidity problem is transferred to the normal part of the biharmonic equation.
This is the main difficulty in the exceptional case.
Furthermore, since it has been shown that four-dimensional biharmonic hypersurfaces in $\mathbb{E}^5$ are minimal, the first remaining case for four-dimensional pseudo-umbilical biharmonic submanifolds is codimension two.
In this paper, we overcome these difficulties in codimension two and prove a stronger statement than is required for the biharmonic problem.
\begin{remark}
A situation similar to that of pseudo-umbilical submanifolds occurs for hypersurfaces. 
If $M$ is a hypersurface in $N^n(c)$, then the tangential part can be rewritten in a more manageable form, which can be used effectively to facilitate the proof of constancy of mean curvature. 
However, in higher codimension, this simplification cannot be used, making the problem considerably more difficult.
Hence studying the normal equation alone is also important for progress towards a resolution of Chen's conjecture. 
\end{remark}
Loubeau and Montaldo introduced $\lambda$-biminimal submanifolds \cite{LM2008}. 
Inoguchi and Lee classified 0-biminimal curves in two-dimensional space forms \cite{IL2012}. 
Sasahara classified $0$-biminimal Lagrangian surfaces in complex space forms under the condition that the integral curves of the Maslov vector field are geodesics  \cite{Sasahara2010}. 
For $\lambda\geq0$, the author classified properly immersed $\lambda$-biminimal submanifolds in $\mathbb{E}^n$ (cf. \cite{Maeta2012}, see also \cite{Maeta2014}).

For $\lambda\in\mathbb{R}$, a submanifold  $M^m$ in $N^n(c)$ is $\lambda$-biminimal if and only if 
\begin{align*}
\Delta^\perp{\bf H}-\sum_{i=1}^mB(A_{\bf H}e_i,e_i)+mc{\bf H}=\lambda{\bf H}.
\end{align*}
For $\lambda=0$, this is precisely the normal part of the biharmonic equation, with no tangential equation imposed. 
Allowing $\lambda\neq0$ gives a further generalization.

We first consider pseudo-umbilical $\lambda$-biminimal submanifolds of codimension two in $\mathbb{E}^n$. 
The tangential equation $(m-4)\nabla f^2=0$, which yields constancy of $f$ when $m\neq4$ in the biharmonic case, cannot be used for general $\lambda$-biminimal submanifolds.
We nevertheless prove that every such connected $\lambda$-biminimal submanifold $M^m$ with $m\geq3$ has constant mean curvature. 
Furthermore, when $\lambda\geq0$, we show that the submanifold is minimal. 
When $\lambda<0$ and $M$ is not minimal in $\mathbb{E}^{m+2}$, then the submanifold is minimal in some hypersphere.
\begin{theorem}\label{main}
Let $M^m$ $(m\geq3)$ be a connected pseudo-umbilical $\lambda$-biminimal submanifold in $\mathbb{E}^{m+2}$. Then the mean curvature $|{\bf H}|$ is constant.
Moreover, if $\lambda\geq0$, then $M$ is minimal.
If $\lambda<0$ and $M$ is not minimal, then 
$|{\bf H}|^2=-\frac{\lambda}{m}$.
In this case, there exists $c\in\mathbb{E}^{m+2}$ such that $M$ is minimal in the hypersphere 
\[
\mathbb{S}^{m+1}
\Big(
c,\sqrt{-\frac{m}{\lambda}}
\Big)\subset \mathbb{E}^{m+2},
\]
where $c$ is the center and $\sqrt{-\frac{m}{\lambda}}$ is the radius of the hypersphere.
\end{theorem}

For Riemannian manifolds of constant sectional curvature, we prove the following result for pseudo-umbilical $0$-biminimal submanifolds of codimension two.

\begin{theorem}\label{sub}
Let $M^m$ $(m\geq3)$ be a connected pseudo-umbilical 0-biminimal submanifold in $N^{m+2}(c)$. Then $|{\bf H}|$ is constant.
Moreover, if $c\leq0$, then $M$ is minimal.
If $M$ is not minimal, then 
$\nabla^\perp{\bf H}=0$ and $|{\bf H}|^2=c>0$.
In the unit sphere, every such non-minimal submanifold is minimal in a small hypersphere $\mathbb{S}^{m+1}\left(\frac{1}{\sqrt{2}}\right)\subset \mathbb{S}^{m+2}$.
\end{theorem}

The theorems show that, in codimension two and dimension at least three, the normal part of the biharmonic equation alone forces a pseudo-umbilical submanifold to be minimal in Euclidean or hyperbolic space, and to have constant mean curvature in a sphere.

Furthermore, as a corollary, we obtain a partial affirmative answer to Chen's conjecture. 

\begin{corollary}\label{cfour}
Every pseudo-umbilical biharmonic submanifold of dimension four in $\mathbb{E}^6$
 is minimal.
\end{corollary}

Moreover, we also obtain a partial affirmative answer to generalized Chen's conjecture in constant non-positive sectional curvature ``Any biharmonic submanifold in a Riemannian manifold of constant non-positive sectional curvature is minimal", and to the Balmu\c{s}-Montaldo-Oniciuc conjecture ``Any biharmonic submanifold in $\mathbb{S}^{n}$ has constant mean curvature" (cf. \cite{BMO2008}). 

\begin{corollary}\label{cfourhyp}
Every pseudo-umbilical biharmonic submanifold of dimension four in a six-dimensional Riemannian manifold of constant non-positive sectional curvature is minimal.
\end{corollary}

\begin{corollary}\label{cfoursph}
Every pseudo-umbilical biharmonic submanifold of dimension four in $\mathbb{S}^6$
 has constant mean curvature on each connected component.
\end{corollary}

Combining Corollary~\ref{cfour} with Dimitri\'c's result for dimensions other than four, we also conclude that every pseudo-umbilical biharmonic submanifold in $\mathbb{E}^6$ is minimal.
Combining Corollary~\ref{cfour}, Corollary~\ref{cfourhyp}, Dimitri\'c's result for dimensions other than four, and Caddeo, Montaldo, and Oniciuc's result for dimensions other than four, we also show that every pseudo-umbilical biharmonic submanifold in a six-dimensional Riemannian manifold of constant non-positive sectional curvature is minimal.
Moreover, Corollary \ref{cfoursph} and Balmu\c{s}, Montaldo, and Oniciuc's result for dimensions other than four prove that every pseudo-umbilical biharmonic submanifold in $\mathbb{S}^6$ has constant mean curvature on each connected component.

The proofs of Theorems \ref{main} and \ref{sub} are inspired by the work of Brinkmann \cite{Brinkmann1925} and Tashiro \cite{Tashiro1965} (see also pages 129 and 212 in \cite{Petersen2016}, \cite{Maeta2021}).
The idea is as follows. On the open set where $f>0$ and $df\neq0$, the assumption of codimension two allows us to choose an adapted local orthonormal frame $\{\xi=\frac{\bf H}{f},\nu\}$ of the normal bundle.
The Gauss, Codazzi, and Ricci equations impose a strong restriction on the shape operator $A_\nu$. 
These equations show that the local metric is a warped product.
By the $\lambda$-biminimal equation, we have ODEs and a polynomial relation. 
Finally, using the properties of the resultant, we conclude that the mean curvature is constant.
Once the constancy of mean curvature has been obtained, a classical theorem of Chen \cite{Chen1971a,Chen1971b} and the $\lambda$-biminimal equation complete the proof of Theorem \ref{main}.  
Theorem \ref{sub} follows from the corresponding constant mean curvature results in a Riemannian manifold of constant sectional curvature by Oniciuc (cf. \cite{Oniciuc2002}, \cite{Oniciuc2003}, see also \cite{BMO2008}).

\section{Preliminaries}

Let ${\bf x}:M^m\rightarrow N^n(c)$ be an $m$-dimensional submanifold in a Riemannian manifold $N^n(c)$ of constant sectional curvature $c$, and let $g$ be the induced metric on $M$.
For any vector fields $X,Y\in \mathfrak{X}(M)$ and a normal vector field $\xi$, the Gauss and Weingarten formulas are given by
\[
\begin{aligned}
& \bar \nabla_XY=\nabla _XY+B(X,Y),\\
& \bar \nabla _X\xi =-A_\xi X+\nabla^\perp_X\xi,
\end{aligned}
\]
where $\nabla$ and $\bar\nabla$ are the Levi-Civita connections of $M$ and $N^n(c)$, respectively.
It is well known that $A$ and $B$ are related by
\[
\langle B(X,Y),\xi\rangle =\langle A_\xi X,Y \rangle.
\]
The mean curvature vector is defined by
\[
{\bf H}=\frac{1}{m}\sum_{i=1}^mB(e_i,e_i).
\]

From now on, assume that $m\geq3$, $n=m+2$, and $M$ is pseudo-umbilical.
Set $f=|{\bf H}|$ and define
\[
U=
\left\{
p\in M ~\mid~ f(p)>0 \quad \text{and} \quad df_p\neq0
\right\}.
\]
We will prove $U$ is empty. 

Assume that $U$ is nonempty. 
Unless otherwise stated, the following calculations are carried out on sufficiently small open subsets of $U$. 
Since $M$ is pseudo-umbilical, $A_{\bf H}X=|{\bf H}|^2X$.
Set $\xi=\frac{\bf H}{f}$. We have
\begin{equation}\label{Axi}
A_{\xi}X=f X.
\end{equation}
Let $\{\xi,\nu\}$ be a local orthonormal frame of the normal bundle.
 Since
 \begin{align*}
 {\rm tr}A_\nu
 =\sum_{i=1}^m\langle B(e_i,e_i),\nu\rangle 
 =\langle m{\bf H},\nu \rangle 
 =\langle mf\xi,\nu \rangle =0,
  \end{align*}
$A_\nu$ is trace-free.
By a direct calculation, we can write
\begin{align}\label{alpha}
\nabla^\perp_X\xi=\alpha(X)\nu,\\
\nabla^\perp_X\nu=-\alpha(X)\xi,
\end{align}   
for some 1-form $\alpha$.
Hence the second fundamental form can be written 
\begin{equation}\label{2nd}
B(X,Y)=f\langle X,Y\rangle \xi+\langle A_\nu X,Y\rangle \nu.
\end{equation}

\begin{lemma}
The 1-form $\alpha$ is closed.
\end{lemma}

\begin{proof}
By \eqref{alpha} and the definition of the normal curvature, we have
\begin{align*}
R^{\perp}(X,Y)\xi
&=\nabla^\perp_X\nabla^\perp_Y\xi-\nabla^\perp_Y\nabla^\perp_X\xi-\nabla^\perp_{[X,Y]}\xi\\
&=\{X(\alpha(Y))-Y(\alpha(X))-\alpha([X,Y])\}\nu\\
&=d\alpha(X,Y)\nu.
\end{align*}
By the Ricci equation, one has
\begin{align*}
d\alpha(X,Y)
&=\langle R^\perp (X,Y)\xi,\nu \rangle \\
&=\langle [A_\xi,A_\nu]X,Y\rangle \\
&=\langle [f {\rm Id},A_\nu]X,Y\rangle =0,
\end{align*}
where the third equality follows from \eqref{Axi}.
\end{proof}

By the Codazzi equation, we obtain the following lemma.

\begin{lemma}
For any smooth vector fields $X,Y$ on $U$,
\begin{equation}\label{xicomp}
X(f)Y-Y(f)X=\alpha(X)A_\nu Y-\alpha(Y)A_\nu X,
\end{equation}
\begin{equation}\label{nucomp}
(\nabla_XA_\nu)Y-(\nabla_YA_\nu)X=f\{\alpha(Y)X-\alpha(X)Y\}.
\end{equation}
\end{lemma}
\begin{proof}
By \eqref{2nd}, 
\begin{align*}
\nabla^\perp_X(B(Y,Z))
&= \nabla ^\perp_X \{f\langle Y,Z\rangle \xi+\langle A_\nu Y,Z\rangle \nu\}\\
&= (Xf)\langle Y,Z\rangle\xi+fX\langle Y,Z\rangle \xi\\
&\qquad+ f\alpha(X)\langle Y,Z\rangle \nu\\
&\qquad+X\langle A_\nu Y,Z\rangle \nu-\alpha(X)\langle A_\nu Y,Z\rangle\xi.  
\end{align*}
Since 
\begin{align*}
B(\nabla_XY,Z)+B(Y,\nabla _XZ)
=fX\langle Y,Z\rangle\xi+\langle A_\nu(\nabla_XY),Z\rangle\nu+\langle A_\nu Y,\nabla _XZ\rangle\nu,   
\end{align*}
we obtain 
\begin{align*}
(\nabla^\perp_X B)(Y,Z)
&=\nabla^\perp_X(B(Y,Z))-B(\nabla_XY,Z)-B(Y,\nabla _XZ)\\
&= (Xf)\langle Y,Z\rangle\xi-\alpha(X)\langle A_\nu Y,Z\rangle\xi\\
&\qquad+ f\alpha(X)\langle Y,Z\rangle \nu+X\langle A_\nu Y,Z\rangle \nu\\
&\qquad-\langle A_\nu(\nabla_XY),Z\rangle\nu-\langle A_\nu Y,\nabla _XZ\rangle\nu.  
\end{align*}
By the Codazzi equation
\[
(\nabla^\perp_X B)(Y,Z)-(\nabla^\perp_Y B)(X,Z)=0,
\]
the $\xi$-component is
\begin{align*}
(Xf)\langle Y,Z\rangle-\alpha(X)\langle A_\nu Y,Z\rangle
-\{(Yf)\langle X,Z\rangle-\alpha(Y)\langle A_\nu X,Z\rangle\}=0.
\end{align*}
This leads to \eqref{xicomp}.
The $\nu$-component of the Codazzi equation is
\begin{align*}
0&=f\alpha(X)\langle Y,Z\rangle +X\langle A_\nu Y,Z\rangle 
-\langle A_\nu(\nabla_XY),Z\rangle-\langle A_\nu Y,\nabla _XZ\rangle\\
&\quad-\{
f\alpha(Y)\langle X,Z\rangle +Y\langle A_\nu X,Z\rangle 
-\langle A_\nu(\nabla_YX),Z\rangle-\langle A_\nu X,\nabla _YZ\rangle
\}\\
&=f\alpha(X)\langle Y,Z\rangle-f\alpha(Y)\langle X,Z\rangle \\
&\quad+\langle \nabla _X(A_\nu Y),Z\rangle -\langle A_\nu(\nabla _XY),Z \rangle\\
&\quad-\langle \nabla _Y(A_\nu X),Z\rangle +\langle A_\nu(\nabla _YX),Z \rangle\\
&=f\alpha(X)\langle Y,Z\rangle-f\alpha(Y)\langle X,Z\rangle \\
&\quad+\langle (\nabla _XA_\nu)(Y),Z\rangle
-\langle (\nabla _YA_\nu)(X),Z\rangle.
\end{align*}
This leads to \eqref{nucomp}.
\end{proof}

The closed 1-form $\alpha$ does not vanish on $U$.
\begin{lemma}\label{l5.1}
The 1-form $\alpha$ satisfies $\alpha\neq0$ on $U$.
\end{lemma}

\begin{proof}
Assume that $\alpha=0$ at $p\in U$.
By \eqref{xicomp}, we have
\[
X(f)Y-Y(f)X=0
\]
at $p$.
Since $m\geq2$, we may assume that 
$\langle X,Y\rangle =0$ and $|X|=|Y|=1$.
Hence we obtain
$X(f)=0$ at $p$. Since $X$ is arbitrary, $df=0$ at $p\in U$, which is a contradiction. 
\end{proof}

By Lemma \ref{l5.1}, we have the following lemma.

\begin{lemma}\label{l5.2}
For $m\geq3$, $df\wedge \alpha=0$ on $U$.
\end{lemma}

\begin{proof}
Since $\alpha\neq0$, we can choose an orthonormal frame $\{e_1,\cdots,e_m\}$ and dual frame $\{\theta^1,\cdots,\theta^m\}$ such that
\[
\alpha=a\theta^1, \qquad a\neq0.
\]
Hence 
\[
{\rm Ker }~\alpha={\rm span}\{e_2,\cdots,e_m\}.
\]
Since $m\geq3$, we can take $\gamma,\delta\in\{2,\cdots,m\}$ with $\gamma\neq\delta$.
By \eqref{xicomp},
\[
e_\gamma(f)e_\delta-e_\delta(f)e_\gamma=\alpha(e_\gamma)A_\nu e_\delta-\alpha(e_\delta)A_\nu e_\gamma=0,
\] 
where we used $e_\gamma,e_\delta\in {\rm Ker}~\alpha$.
Therefore, we have $e_\gamma(f)=0$ for any $\gamma\geq2$.

Set $X=\sum_{i=1}^mX_ie_i$. Since
\[
df(X)=X_1df(e_1)~~\text{and}~~\alpha(X)=aX_1,
\]
we obtain 
\[
(df\wedge \alpha)(X,Y)
=df(X)\alpha(Y)-df(Y)\alpha(X)=0.
\]
\end{proof}

On $U$, we can choose an adapted orthonormal frame.

\begin{lemma}\label{l5.3}
On every sufficiently small open subset of $U$, there exist a local orthonormal frame $\{e_1,\cdots,e_m\}$ and  smooth functions $a,b$ satisfying $a,b\neq0$ on this subset, such that 
\[
e_1=\frac{\nabla f}{|\nabla f|}, \quad \alpha=a\theta^1,
\]
\[
e_1(f)=ab, \quad e_\gamma(f)=0 \quad (\gamma=2,3,\cdots, m),
\]
and 
\begin{align*}
&A_\nu e_1=-(m-1)be_1,\\
&A_\nu e_\gamma=be_\gamma \quad (\gamma=2,3,\cdots,m).
\end{align*}

\end{lemma}

\begin{proof}
By the proof of Lemma \ref{l5.2}, $\alpha=a\theta^1$ and $e_\gamma(f)=0$ for $\gamma=2,3,\cdots,m.$ 
Hence we can take 
\[
e_1=\frac{\nabla f}{|\nabla f|}.
\]

Substituting $X=e_1$ and $Y=e_\gamma$ into \eqref{xicomp}, 
\[
e_1(f)e_\gamma=aA_\nu e_\gamma.
\]
Set $b=\frac{e_1 (f)}{a}\neq0$. Then we have
\[
A_\nu e_\gamma=be_\gamma, \quad e_1(f)=ab.
\]
Set $A_\nu e_1=\sum_{p=1}^m\beta_pe_p$. Since 
\[
0=\langle A_\nu e_\gamma,e_1 \rangle =\langle e_\gamma, A_\nu e_1 \rangle=\beta_\gamma, 
\]
\[
A_\nu e_1=\beta_1 e_1.
\]
Since $A_\nu$ is trace-free, we have
$\beta_1=-(m-1)b$. 

\end{proof}

Let $\mathcal{D}={\rm Ker}~df.$ Since $\mathcal{D}$ is tangent to the level set of $f$, it is integrable, that is, for any vector fields $X,Y\in\mathcal{D}$, $[X,Y]\in\mathcal{D}$.

\begin{lemma}\label{l6.1}
For any vector field $X\in\mathcal{D}$, $Xb=0$.
\end{lemma}

\begin{proof}
A vector field $X$ can be written as $X=\sum_{i=2}^mX_ie_i$. We have
\begin{equation*}
A_\nu X=\sum_{i=2}^mX_ibe_i=bX
\end{equation*}
and 
\[
\alpha(X)=\sum_{i=2}^mX_ia\theta^1(e_i)=0.
\]
Let $X$ and $Y$ be smooth vector fields tangent to $\mathcal{D}$. By \eqref{nucomp}, 
\begin{align*}
0
&= (\nabla_X A_\nu)Y-(\nabla_Y A_\nu)X\\
&= \nabla_X (A_\nu Y)-A_\nu(\nabla_XY)
-\{\nabla_Y (A_\nu X)-A_\nu(\nabla_YX)\}\\
&=(Xb)Y-(Yb)X+b([X,Y])-A_\nu([X,Y]).
\end{align*}
Since $[X,Y]\in\mathcal{D}$, $A_\nu([X,Y])=b[X,Y]$. Hence we have
\[
(Xb)Y-(Yb)X=0.
\]
Since $m\geq3$, we obtain $Xb=0$.

\end{proof}

By the above lemmas, we obtain a system of ODEs for $m$-dimensional pseudo-umbilical submanifolds in $N^{m+2}(c)$ $(m\geq3)$.

\begin{lemma}\label{l6.2}
There exists a smooth function $\sigma$ such that 
\[
\nabla _Ye_1=\sigma Y, \quad \text{for every smooth vector field}\quad Y\in\mathcal{D},
\]
\[
\nabla_{e_1}e_1=0,
\]
and 
\[
e_1 b+m\sigma b=-fa.
\]
\end{lemma}

\begin{proof}
By \eqref{nucomp}, for $Y\in\mathcal{D}$, we have
\begin{equation}\label{221}
(\nabla_{e_1}A_\nu)Y-(\nabla _YA_\nu)e_1
=-faY.
\end{equation}
By Lemma \ref{l6.1}, 
\begin{equation}\label{230}
(\nabla_{e_1}A_\nu)Y=(e_1 b)Y+b\nabla_{e_1}Y-A_\nu(\nabla_{e_1}Y)
\end{equation}
and
\begin{align}\label{231}
(\nabla _YA_\nu)e_1
&=\nabla_Y(A_\nu e_1)-A_\nu(\nabla_Ye_1)\\
&=-(m-1)(Yb)e_1-(m-1)b\nabla_Ye_1-A_\nu(\nabla_Ye_1)\notag\\
&=-(m-1)b\nabla_Ye_1-A_\nu(\nabla_Ye_1).\notag
\end{align}

Since 
\[
\langle \nabla_Ye_1,e_1\rangle =0,
\]
$\nabla_Ye_1\in \mathcal{D}.$
Hence 
\[
A_\nu(\nabla_Ye_1)=b\nabla_Y e_1.
\]
Substituting this into \eqref{231}, 
\[
(\nabla _YA_\nu)e_1=-mb\nabla_Ye_1.
\]
We consider the $\mathcal{D}$-component of \eqref{230}.
For $Z\in\mathcal{D}$,
\begin{align*}
\langle (\nabla_{e_1}A_\nu)Y,Z\rangle 
&=\langle (e_1 b)Y+b\nabla_{e_1}Y-A_\nu(\nabla_{e_1}Y),Z\rangle\\
&=\langle (e_1 b)Y,Z\rangle +\langle \nabla_{e_1}Y,bZ-A_\nu Z\rangle \\
&=\langle (e_1 b)Y,Z\rangle.
\end{align*}
Hence the $\mathcal{D}$-component of $(\nabla_{e_1}A_\nu)Y$ is $(e_1 b)Y$.
Therefore, the $\mathcal{D}$-component of \eqref{221} is
\[
(e_1 b)Y+mb\nabla _Ye_1=-faY,
\]
from which we deduce
\[
\nabla_Ye_1=\sigma Y.
\]

The $e_1$-component of $(\nabla_{e_1}A_\nu)Y$ is
\[
b\langle\nabla_{e_1}Y,e_1 \rangle
-\langle \nabla_{e_1}Y,A_\nu e_1
\rangle   
\]
and by taking the $e_1$-component of \eqref{221}, we have 
\[
\langle Y,\nabla _{e_1}e_1\rangle=0.
\]
Hence the $\mathcal{D}$-component of $\nabla_{e_1}e_1$ vanishes.
Combining this with 
\[
\langle \nabla_{e_1}e_1,e_1\rangle =0,
\]
we obtain 
\[
\nabla_{e_1}e_1=0.
\]

\end{proof}

By the above lemmas, we have a local warped product structure on $U$.
\begin{lemma}
Near every point of $U$, there exist local coordinates in which the metric $g$ takes the form
\[
g=ds^2+\varphi(s)^2g_L,\qquad e_1=\partial_s,
\]
where $s$ is an arclength parameter along the integral curves of $e_1$, $L$ is a local level leaf of $f$, $g_L$ is a metric on $L$, $\varphi$ is a positive smooth function of $s$, and
\[
\sigma=\frac{\partial_s\varphi}{\varphi}.
\] 
Furthermore, $f,a,b,\sigma$ depend only on $s$.
\end{lemma}

\begin{proof}
For $X,Y\in\mathcal{D}$, we have
\begin{align*}
d\theta^1(X,Y)
&=(\nabla_X\theta^1)(Y)-(\nabla_Y\theta^1)(X)\\
&=-\theta^1(\nabla_XY)+\theta^1(\nabla_YX)\\
&=\theta^1([Y,X])=0,
\end{align*}
where the second equality follows from $\theta^1(Y)=0$ and the final equality follows from the integrability of $\mathcal{D}$.
A similar calculation, using Lemma \ref{l6.2}, gives
\[
d\theta^1(e_1,Y)=0.
\]
Hence $\theta^1$ is closed. By the Poincar\'e lemma, there exists $s$ such that $\theta^1=ds$ and $e_1=\partial_s$, where $s$ is an arclength parameter of an integral curve of $e_1$. 
Thus $df=\partial_sfds$ and $f$ depends only on $s$.
By Lemma \ref{l6.1}, $b$ depends only on $s$.
By Lemma \ref{l5.3}, we have
\[
a=\frac{\partial_s f}{b},
\]
and $a$ depends only on $s$.
By Lemma \ref{l6.2}, 
\[
\sigma=-\frac{\partial_sb+fa}{mb}.
\]
Hence $\sigma$ also depends only on $s$. 

Let $Y,Z$ be tangent vector fields on a fixed leaf.
Extend them locally by the flow of $e_1$ so that $[e_1,Y]=[e_1,Z]=0.$
Then we have
\begin{align*}
\partial_sg(Y,Z)
&=\langle \nabla_{e_1}Y,Z\rangle+\langle Y,\nabla_{e_1}Z\rangle\\
&=\langle \nabla_Ye_1,Z\rangle+\langle Y,\nabla_Ze_1\rangle\\  
&=2\sigma g(Y,Z).
\end{align*}
Let $g_s$ denote the metric induced on the level set $s=\text{const}$. Then
\[
\partial_s g_s=2\sigma(s)g_s.
\]
Therefore, for a metric $g_L$ on one leaf and a positive smooth function $\varphi$, $g_s$ is written as
\[
g_s=\varphi(s)^2g_L, \quad \sigma=\frac{\partial_s\varphi}{\varphi}.
\]
\end{proof}

\section{ODEs for pseudo-umbilical $\lambda$-biminimal submanifolds}

In this section, we will deduce ODEs from the $\lambda$-biminimal equation.
Since $M$ is a pseudo-umbilical submanifold, we have
\[
\sum_{i=1}^mB(A_{\bf H}e_i,e_i)=m f^2{\bf H}.
\]
Hence the $\lambda$-biminimal equation 
\[
\Delta^\perp{\bf H}-\sum_{i=1}^mB(A_{\bf H}e_i,e_i)+mc{\bf H}=\lambda{\bf H}
\]
can be written as 
\begin{equation}\label{plb}
\Delta^\perp{\bf H}=(mf^2+\mu){\bf H},
\end{equation}
where $\mu=\lambda-mc$.

In this case, the equation of $\lambda$-biminimal submanifolds is as follows.

\begin{proposition}\label{7.2}
Let 
$
{\bf x}:(M^m,g)\rightarrow N^{m+2}(c)
$
be an $m$-dimensional connected pseudo-umbilical submanifold in an $(m+2)$-dimensional Riemannian manifold of constant sectional curvature $c$.
Then locally on $U$, the $\lambda$-biminimal equation can be written as
\begin{align}
&\Delta f=f(|\alpha|^2+mf^2+\mu),\label{nsn}\\
&2\alpha(\nabla f)+f{\rm div}(\alpha^\sharp)=0.\label{nst}
\end{align}
\end{proposition}

\begin{proof}
Since ${\bf H}=f\xi$, we have
\begin{align*}
\Delta^\perp{\bf H}
&=\Delta^\perp (f\xi)\\
&=
\sum_{i=1}^m
\{
e_ie_if\xi+2e_i f\nabla^\perp_{e_i}\xi+f\nabla^\perp_{e_i}\nabla^\perp_{e_i}\xi
-(\nabla_{e_i}e_i f)\xi-f\nabla^\perp_{\nabla_{e_i}e_i}\xi
\}\\
&=\Delta f\xi+2\alpha(\nabla f)\nu
+f\sum_{i=1}^m\{e_i(\alpha(e_i))\nu-\alpha(\nabla _{e_i}e_i)\nu-(\alpha(e_i))^2\xi\}\\
&=\Delta f\xi+2\alpha(\nabla f)\nu
+f{\rm div}(\alpha^\sharp)\nu-f|\alpha|^2\xi.
\end{align*}
Substituting this into \eqref{plb}, we complete the proof.
\end{proof}

Using the warped product expression for the metric, we obtain the following ODEs.

\begin{proposition}\label{p8.1}
On every sufficiently small connected open subset of $U$, we have
\begin{equation}\label{8.1}
\partial_sf=ab,
\end{equation}
\begin{equation}\label{8.2}
\partial_sa=-(m-1)\sigma a-2\frac{\partial_sf}{f}a,
\end{equation}
\begin{equation}\label{8.3}
\partial_sb=-m\sigma b-fa,
\end{equation}
\begin{equation}\label{8.4}
\partial_s\sigma =-\sigma^2-f^2+(m-1)b^2-c,
\end{equation}
\begin{equation}\label{8.5}
\partial_s\partial_sf+(m-1)\sigma \partial_sf=fa^2+mf^3+\mu f ,
\end{equation}
\begin{equation}\label{46}
-m\sigma\frac{\partial_sf}{f}-2\left(\frac{\partial_sf}{f}\right)^2-2a^2=mf^2+\mu.
\end{equation}
\end{proposition}

\begin{proof}
By Lemmas \ref{l5.3} and \ref{l6.2}, we have \eqref{8.1} and \eqref{8.3}.
\[
\alpha^\sharp=ae_1 \quad \text{and} \quad {\rm div}\,e_1=(m-1)\sigma,
\] 
where {\rm div} denotes the divergence with respect to $g$.
In fact, since $\alpha=a(s)ds$ and $\alpha(X)=\langle \alpha^\sharp ,X\rangle$, 
\[
a(s)ds(X)=\alpha(X)=\langle \alpha^\sharp ,X\rangle=ds(\alpha^\sharp)ds(X).
\]
Hence we have the first equality.
By Lemma \ref{l6.2},
\[
{\rm div}\,e_1=\sum_{i=1}^m\langle \nabla_{e_i}e_1,e_i\rangle =(m-1)\sigma.
\]

By using these equations, we have
\begin{align*}
{\rm div}\,(\alpha^\sharp)
&= {\rm div}\,(ae_1)\\
&= \langle \nabla_{e_1}(ae_1),e_1\rangle+\sum_{i=2}^m\langle \nabla_{e_i}(ae_1),e_i\rangle \\
&=\partial_sa+a\,{\rm div}\,e_1\\
&=\partial_sa+(m-1)a\sigma.
\end{align*}
Substituting this and
\[
\langle \nabla f,\alpha^\sharp \rangle =a\partial_sf
\]
into \eqref{nst}, we obtain \eqref{8.2}.

For $\gamma\in\{2,3,\cdots,m\}$, by \eqref{2nd}, Lemma \ref{l5.3}, and the Gauss equation, the sectional curvature $K(e_1,e_\gamma)$ of the plane spanned by $e_1$ and $e_\gamma$ is
\begin{align}\label{421}
K(e_1,e_\gamma)
&=c+\langle B(e_1,e_1),B(e_\gamma,e_\gamma)\rangle-|B(e_1,e_\gamma)|^2\\
&=c+\langle f\xi-(m-1)b\nu,f\xi+b\nu\rangle  \notag\\
&=c+f^2-(m-1)b^2.\notag
\end{align}
On the other hand, by the curvature equation of the warped product metric (cf.  \cite{ONeill1983}), 
\begin{equation}
K(e_1,e_\gamma)\label{422}
=-\frac{\partial_s\partial_s\varphi}{\varphi}
=-(\partial_s\sigma +\sigma^2).
\end{equation}
Combining \eqref{421} with \eqref{422}, we have \eqref{8.4}.

We consider \eqref{nsn}. In this metric, a straightforward computation gives 
\begin{align*}
\Delta f
=\partial_s\partial_sf+(m-1)\sigma\partial_s f.
\end{align*}
Substituting this into \eqref{nsn} gives \eqref{8.5}.

By \eqref{8.1}, \eqref{8.2}, and \eqref{8.3} we have
\begin{align*}
\partial_s\partial_sf
&=\partial_sab+a\partial_sb\\
&=\left\{-(m-1)\sigma a-2\frac{\partial_sf}{f}a\right\}b+a\{-m\sigma b-fa\}\\
&=-(2m-1)\sigma \partial_sf-2\frac{(\partial_sf)^2}{f}-fa^2.
\end{align*} 
Substituting this into \eqref{8.5} gives \eqref{46}.
\end{proof}

For convenience, set 
\[
p=\frac{b}{f}, \qquad q=\frac{a}{f},\qquad r=\frac{\sigma}{f},\quad \ell=\frac{\mu}{f^2},\quad\text{and}\quad \kappa=\frac{c}{f^2}. 
\]
We introduce a new parameter 
\[
\tau=\tau_0+\int_{s_0}^sf(u) du.
\]

A straightforward computation shows that 

\begin{align*}
\partial_\tau p
=-mrp-q(1+p^2),
\end{align*}
\begin{align*}
\partial_\tau q
=-(m-1)rq-3q^2p,
\end{align*}
\begin{align*}
\partial_\tau r
=-r^2-1+(m-1)p^2-pqr-\kappa,
\end{align*}
\begin{align*}
\partial_\tau f
=pqf.
\end{align*}
By \eqref{46} and \eqref{8.1}, we obtain
\begin{align}\label{ellpqr}
\ell=-mrpq-2p^2q^2-2q^2-m.
\end{align}
By the definition of $\ell$, 
\begin{align}\label{elltau}
\partial_\tau \ell
=-2pq\ell.
\end{align}
We remark that $pq\neq0$.

We introduce the logarithmic mean curvature 
\[
\rho=\ln f,
\]
and use it as a new parameter.
Note that 
\[
\frac{d\rho}{d\tau}=\frac{d}{d\tau}\ln f=\frac{\partial_{\tau}f}{f}=pq\neq0.
\]

Set 
\[
y=1+\frac{1}{p^2},\qquad t=\frac{1}{q^2},\quad \text{and}\quad z=\frac{r}{pq}.
\]
By definition, $y>1$ and $t>0$.

\begin{lemma}\label{l11.1}
We have
\begin{align*}
y'&=2(y-1)(mz+y),\\
t'&=2t\{(m-1)z+3\},\\
z'&=2(m-1)z^2+(y+2)z+(m-y)t-\kappa(y-1)t,\\
\ell&=-\frac{mz+2y+m(y-1)t}{(y-1)t},\\
\ell'&=-2\ell,
\end{align*}
where $'=\partial_\rho$.
\end{lemma}

\begin{proof}
Since 
\[
\partial_\tau y=2mrp^{-2}+2p^{-3}q(1+p^2),
\]
we have
\begin{align*}
y'
=\frac{1}{pq}\partial_\tau y
=2(y-1)(mz+y).
\end{align*}
Substituting 
\[
\partial_\tau t=2(m-1)rq^{-2}+6pq^{-1}
\]
into $t'=\frac{\partial_\tau t}{pq}$, we obtain 
\[
t'=2t\{(m-1)z+3\}.
\]
A long but straightforward computation shows
\[
z'=2(m-1)z^2+(y+2)z+(m-y)t-\kappa(y-1)t.
\]
Substituting
\[
p^2=\frac{1}{y-1},\qquad q^2=t^{-1},\quad\text{and}\quad r=pqz
\]
into \eqref{ellpqr}, we have
\[
\ell=-\frac{mz+2y+m(y-1)t}{(y-1)t}.
\]
By \eqref{elltau}, we have
\[
\ell'=\frac{\partial_\tau \ell}{pq}=-2\ell.
\]
Also, 
\[
\kappa'=-2\kappa.
\]
\end{proof}

The above ODEs yield the following identity.

\begin{proposition}\label{p12.1}
On every sufficiently small connected open subset of $U$, the following equality holds
\begin{align*}
F_0(y,t,z,\kappa)
&:=2m^2z^2+\{(5m-4)y+6m\}z+12y\\
&\qquad-m(m+y-2)t+m\kappa(y-1)t=0.
\end{align*}
\end{proposition}

\begin{proof}
For convenience, set 
\begin{align*}
N&=mz+2y+m(y-1)t,\\
D&=(y-1)t.
\end{align*}

By Lemma \ref{l11.1}, we have
\begin{equation}\label{ND}
N'D-ND'+2ND=0.
\end{equation}
By the definition of $N$ and $D$,
\[
N'=mz'+2y'+my't+m(y-1)t'
\]
and 
\[
D'=y't+(y-1)t'.
\]
Substituting these into \eqref{ND}, we obtain
\begin{align*}
0
&=(y-1)t
\Big(
m\{2(m-1)z^2+(y+2)z+(m-y)t\}-m\kappa(y-1)t\\
&\qquad\qquad\qquad+2(2+mt)(y-1)(mz+y)\\
&\qquad\qquad\qquad+2m(y-1)t\{(m-1)z+3\}\\
&\qquad\qquad\qquad+(mz+2y+m(y-1)t)\{2-2(mz+y)-2((m-1)z+3)\}
\Big)\\
&=-(y-1)t
\Big(
2m^2z^2+\{(5m-4)y+6m\}z\\
&\qquad\qquad\qquad\qquad+(12-mt)y-m(m-2)t+m\kappa(y-1)t
\Big).
\end{align*}
\end{proof}

Assume that $\mu\neq0$. Set 
\[
\beta=-\frac{mc}{\mu}.
\]
We note that, when $c=0$, all the calculations below remain valid for arbitrary $\lambda$ on setting $\beta=0$.
 
We introduce
\[
u=\frac{mz}{y},\qquad v=\frac{m^2t}{y}.
\]
Since $t>0$ and $y>1$, $v>0$.

\begin{lemma}\label{l13.1}
On every sufficiently small connected open subset of $U$, the following equalities hold
\begin{align}
y'&=2y(y-1)(u+1),\label{y'}\\
u'&=\frac{m-y}{m}(2u^2+v)+(4-y)u\label{u'}-\beta(u+2)-\frac{\beta(y-1)}{m}v,\\
v'&=2v\Big\{\frac{m-y}{m}u+4-y\Big\},\label{v'}
\end{align}
and 
\begin{align}\label{Fm}
&F_m(y,u,v)\\
&\quad:=2myu^2+\{(5m-4)y+6m+m\beta\}u+12m+2m\beta\notag\\
&\qquad-\{m+y-2-\beta(y-1)\}v=0.\notag
\end{align}
\end{lemma}

\begin{proof}
By the definition of $u$,
\begin{align*}
u'
&=\frac{m}{y}
\Big\{
2(m-1)z^2+(y+2)z+(m-y)t
-\kappa(y-1)t\Big\}\\
&\qquad-\frac{mz}{y^2}
\Big\{
2(y-1)(mz+y)
\Big\}.
\end{align*}
Substituting 
\begin{equation}\label{zt}
z=\frac{1}{m}uy,\quad t=\frac{1}{m^2}vy,\quad\text{and}\quad \kappa=\beta\frac{m(u+2)}{(y-1)v}+\beta
\end{equation}
into the above equation, 
\[
u'=\frac{m-y}{m}(2u^2+v)+(4-y)u-\beta(u+2)-\frac{\beta(y-1)}{m}v.
\]
By the same argument, 
\begin{align*}
v'
&=\frac{m^2}{y}t'-\frac{m^2t}{y^2}y'\\
&=2v\Big\{\frac{m-y}{m}u+4-y\Big\}.
\end{align*}
By Lemma \ref{l11.1},
\[
y'=2(y-1)(mz+y)=2y(y-1)(u+1).
\]
Substituting \eqref{zt} into $F_0$, we obtain
\begin{align*}
F_0
&=2u^2y^2+\{(5m-4)y+6m\}\frac{y}{m}u
+12y-\frac{y}{m}vy-(m-2)v\frac{y}{m}\\
&\qquad+\frac{y}{m}\{\beta m(u+2)+\beta(y-1)v\}\\
&=\frac{y}{m}\Big\{2myu^2+\{(5m-4)y+6m+\beta m\}u\\
&\qquad+12m+2\beta m-\{m+y-2-\beta(y-1)\}v
\Big\}\\
&=\frac{y}{m}F_m.
\end{align*}
Therefore, by Proposition \ref{p12.1}, we complete the proof.
\end{proof}

Considering Lemma \ref{l13.1}, we introduce the following vector field
\begin{align*}
\mathcal{X}_m
&=2y(y-1)(u+1)\partial_y\\
&\quad+\Big\{\frac{m-y}{m}(2u^2+v)+(4-y)u\\
&\qquad\quad
-\beta(u+2)-\frac{\beta(y-1)}{m}v\Big\}\partial_u\\
&\quad+2v\Big\{\frac{m-y}{m}u+4-y\Big\}\partial_v.
\end{align*}
Any smooth function $G(y,u,v)$ satisfies  
$G'=\mathcal{X}_m G$ along an integral curve of $\mathcal{X}_m$.

Set 
\begin{align*}
\tilde F_m&=2myu^2+\{(5m-4)y+6m+\beta m\}u+12m+2\beta m,\\
\tilde{\mathcal{X}}_m
&=2y(y-1)(u+1)\partial_y\\
&\quad+\Big\{\frac{m-y}{m}(2u^2+v)+(4-y)u
-\beta(u+2)-\frac{\beta(y-1)}{m}v\Big\}\partial_u.
\end{align*}
Set 
\[
d_m=m+y-2-\beta(y-1).
\]
Assume that $d_m\neq0$. We note that, in Theorems \ref{main} and \ref{sub}, $d_m>0$.
Lemma \ref{l13.1} yields
\begin{equation}\label{vtFm}
v=\frac{\tilde F_m}{d_m}.
\end{equation}
Hence we have
\begin{align*}
\tilde{\mathcal{X}}_m
&=2y(y-1)(u+1)\partial_y\\
&\quad
+\Big\{\frac{m-y}{m}
\Big(2u^2+\frac{\tilde F_m}{d_m}\Big)+(4-y)u
-\beta(u+2)-\frac{\beta(y-1)\tilde F_m}{m d_m}\Big\}\partial_u\\
&=2y(y-1)(u+1)\partial_y\\
&\qquad+\frac{1}{m d_m}
\Big\{
D_2u^2+D_1u+D_0\Big\}\partial_u,
\end{align*}
where
\begin{align*}
D_2
&=2\{-(m+1)y^2+(m^2+2)y+m(m-2)\}\\
&\quad
+2\beta
\{(-m+1)y^2-y+m\},
\end{align*}
\begin{align*}
D_1
&=2\{(-3m+2)y^2+2m(m-1)y+m(5m-4)\}\\
&\quad
+4\beta\{(-m+1)y^2-(2m+1)y+3m\},
\end{align*}
\begin{align*}
D_0=-12my+12m^2+16m\beta(-y+1).
\end{align*}
Set 
\begin{align}
P_{0,m}
&=A_3u^3+A_2u^2+A_1u+A_0,\label{P0m}\\
P_{1,m}
&=md_m\tilde{\mathcal{X}}_mP_{0,m},\label{P1m}
\end{align}
where
\begin{align*}
A_3
&=2my^2\Big\{(2m+1)y-3m^2+2m-2+\beta \{(2m-1)y-2m+1\}\Big\},\\
A_2
&=-my
\Big\{(-15m+12)y^2+(22m^2-54m+8)y+11m^2+24m-8\\
&\qquad\quad +\beta\{(-15m+12)y^2+4(m-4)y+m^2+10m+4\}\Big\},\\
A_1
&=2(5m^2-9m+4)y^3+4m(-5m^2+30m-16)y^2\\
&\qquad+6m(-7m^2+8)y+6m^2(m-8)\\
&\quad+
\beta\{
(15m^2-22m+8)y^3+(27m^2+14m-8)y^2\\
&\qquad\qquad+m(-3m^2-32m+8)y+m^2(m-8)
\},\\
A_0
&=6m\{(5m-4)y^2+m(-7m+20)y+2m(m-8)\}\\
&\quad
-2\beta m
\{-4(5m-4)y^2+(m^2+12m-16)y-m(m-8)\}.
\end{align*}

\begin{lemma}\label{l14.1}
On every sufficiently small connected open subset of $U$, the following equalities hold
\[
P_{0,m}(y,u)=0\quad \text{and}\quad P_{1,m}(y,u)=0.
\]
\end{lemma}

\begin{proof}
We calculate $\mathcal{X}_mF_m$.
A straightforward computation shows that 
\begin{align*}
\partial_yF_m
&=2mu^2+(5m-4)u-v+\beta v,\\
\partial_uF_m
&=4myu+(5m-4)y+6m+\beta m,\\
\partial_vF_m
&=-d_m.
\end{align*}
Hence
\begin{align*}
\mathcal{X}_mF_m
&=2y(y-1)(u+1)\{2mu^2+(5m-4)u-v+\beta v\}\\
&\quad+\frac{1}{m d_m}\Big\{D_2u^2+D_1u+D_0\Big\}\{4myu+(5m-4)y+6m+\beta m\}\\
&\quad-2v\Big\{\frac{m-y}{m}u+4-y\Big\}d_m\\
&=\frac{1}{md_m}
\Big\{
2md_my(y-1)(u+1)\{2mu^2+(5m-4)u-v+\beta v\}\\
&\qquad\quad+\Big\{D_2u^2+D_1u+D_0\Big\}\{4myu+(5m-4)y+6m+\beta m\}\\
&\qquad\quad-2\tilde F_m\Big\{(m-y)u+(4-y)m\Big\}d_m
\Big\}\\
&=-\frac{2}{md_m}P_{0,m}(y,u).
\end{align*}
Since $F_m(y(\rho),u(\rho),v(\rho))\equiv0$, we have $\mathcal{X}_mF_m=0$ and $P_{0,m}=0$. 
By the definition of $P_{1,m}$, we also obtain $P_{1,m}=0$.
\end{proof}

We give an expression for $P_{1,m}$, which will be used later.

A long but straightforward computation shows 
\begin{align*}
\partial_y(P_{0,m})&=B_3u^3+B_2u^2+B_1u+B_0,\\
\partial_u(P_{0,m})&=C_2u^2+C_1u+C_0,
\end{align*}
where
\begin{align*}
B_3
&=6m(2m+1)y^2+4m(-3m^2+2m-2)y\\
&\qquad+\beta\{6m(2m-1)y^2-4m(2m-1)y\},
\end{align*}
\begin{align*}
B_2
&=m\{9(5m-4)y^2+4(-11m^2+27m-4)y-11m^2-24m+8\}\\
&\qquad+\beta m\{9(5m-4)y^2-8(m-4)y-m^2-10m-4\},
\end{align*}
\begin{align*}
B_1
&=2\Big\{3(5m^2-9m+4)y^2+4m(-5m^2+30m-16)y\\
&\qquad -3m(7m^2-8)\Big\}\\
&\quad
+\beta \Big\{
3(15m^2-22m+8)y^2+2(27m^2+14m-8)y\\
&\qquad\qquad+m(-3m^2-32m+8)
\Big\},
\end{align*}
\begin{align*}
B_0
&=6m\{2(5m-4)y+m(-7m+20)\}\\
&\qquad
+2\beta m\{
8(5m-4)y-m^2-12m+16
\},
\end{align*}
\[
C_2=3A_3, \qquad C_1=2A_2,\quad \text{and}\quad C_0=A_1.
\]
Hence
\begin{align}\label{P1m2}
P_{1,m}(y,u)
&=md_m\tilde{\mathcal{X}}_m P_{0,m}(y,u)\\
&=2md_my(y-1)(u+1)
\Big\{
B_3u^3+B_2u^2+B_1u+B_0
\Big\}\notag\\
&\qquad+
\{
D_2u^2+D_1u+D_0
\}
\{
C_2u^2+C_1u+C_0
\}\notag.
\end{align}

\section{Resultant}

To show Theorems \ref{main} and \ref{sub}, we will use basic properties of the resultant (cf. \cite{CLO2005}).
\begin{definition}[page 77 in \cite{CLO2005}]
Given two polynomials $f,g\in \mathbb{R}[x]$ of positive degree, say
\begin{align*}
f&=a_\ell x^\ell+\cdots +a_0,\quad a_\ell\neq0,\quad \ell>0,\\
g&=b_mx^m+\cdots +b_0,\quad b_m\neq0,\quad m>0.
\end{align*}
Then the {\em resultant} of $f$ and $g$, denoted ${\rm Res}(f,g)$, is the $(\ell+m)\times(\ell+m)$ determinant
\begin{align*}
{\rm Res}(f,g) ={\rm det} 
\begin{pmatrix}
a_\ell &  & & & b_m &  &  & \\ 
a_{\ell-1} & a_\ell& & & b_{m-1} & b_m &  & \\
a_{\ell-2} & a_{\ell-1}& \ddots & & b_{m-2} & b_{m-1} & \ddots & \\
\vdots & a_{\ell-2}& \ddots& a_\ell& \vdots & b_{m-2} & \ddots &b_m \\
a_0 & \vdots& \ddots& a_{\ell-1}& b_0 & \vdots & \ddots &b_{m-1} \\
    & a_0& & a_{\ell-2}&  & b_0 &  &b_{m-2} \\
    & & \ddots& \vdots&  &  & \ddots & \vdots\\
    & & & a_0&  &  &  &b_0 \\
\end{pmatrix},
\end{align*}
where the first $m$ columns correspond to $f$, and the last $\ell$ columns correspond to $g$.
Here the blank spaces are filled with zeros.
\end{definition}

We recall the standard properties of the resultant.
\begin{lemma}[pages 77--78 in \cite{CLO2005}]\label{lresultant}
Let $P,Q\in \mathbb{R}[y,u]$ be polynomials of positive degree in $u$, and set 
\[
\mathcal{R}(y)={\rm Res}_u(P(y,u),Q(y,u))\in\mathbb{R}[y].
\]
Then the following statements hold.

\begin{enumerate}
\item
If $y_0,u_0\in\mathbb{R}$ satisfy
\[
P(y_0,u_0)=Q(y_0,u_0)=0,
\]
then 
\[
\mathcal{R}(y_0)=0.
\]

\item
Suppose that specialization at $y=y_0$ lowers neither the degree of $P$ nor that of $Q$ in $u$.
Then 
\[
\mathcal{R}(y_0)={\rm Res}_u(P(y_0,u),Q(y_0,u)).
\]
In particular, $\mathcal{R}(y_0)=0$ if and only if $P(y_0,u)$ and $Q(y_0,u)$ have a nontrivial common factor in $\mathbb{R}[u]$.
\end{enumerate}
\end{lemma}

The following lemma is a direct consequence of the above lemma.

\begin{lemma}\label{lYconst}
Let $Y,U:I\rightarrow \mathbb{R}$ be continuous functions on a connected interval $I$.
Assume that $P,Q\in\mathbb{R}[y,u]$ are polynomials in $u$ whose coefficients belong to $\mathbb{R}[y]$, and both $P$ and $Q$ have positive degree in $u$.
Set 
\[
\mathcal{R}(y)={\rm Res}_u(P(y,u),Q(y,u)).
\]
If 
\[
P(Y(\rho),U(\rho))=Q(Y(\rho),U(\rho))=0\qquad(\rho\in I)
\]
and $\mathcal{R}$ is not the zero polynomial, then $Y$ is constant on $I$.
\end{lemma}

\begin{proof}
Fix $\rho\in I$. Set 
\[
y_0=Y(\rho),\qquad u_0=U(\rho).
\] 
By the assumption 
\[
P(y_0,u_0)=Q(y_0,u_0)=0.
\]
Hence (1) of Lemma \ref{lresultant} gives $\mathcal{R}(Y(\rho))=\mathcal{R}(y_0)=0$.
Since $\rho$ is arbitrary, it follows that
\[
\mathcal{R}(Y(\rho))=0
\]
for every $\rho\in I$. 
Therefore, 
\[
Y(I)\subset Z(\mathcal{R})=\{y\in\mathbb{R} ~\mid~ \mathcal{R}(y)=0\}.
\]

Since $\mathcal{R}$ is not the zero polynomial, the set $Z(\mathcal{R})$ is finite.
On the other hand, since $Y$ is continuous on the connected interval $I$, $Y(I)$ is connected.
Hence $Y(I)$ is a connected subset of the finite set $Z(\mathcal{R})$. Therefore $Y(I)$ consists of a single point, and hence $Y$ is constant on $I$.
\end{proof}

\section{Proof of Theorem \ref{main}}

In this section, we show Theorem \ref{main}. 
Since $c=0$, we only have to consider the case $\beta=0$. First, we show that the set $U$ is empty. We divide the proof into two cases.

\subsection{The case $m\neq4$}
We evaluate $P_{0,m}(y,u)$ at $y=1$. By \eqref{P0m},
\begin{align*}
P_{0,m}(1,u)
&=-(m-1)
\{
2m(3m-1)u^3+3m(11m-4)u^2\\
&\qquad\qquad\qquad+2(28m^2-13m+4)u+6m(5m-4)
\}.
\end{align*}
By \eqref{P1m2},
\begin{align*}
P_{1,m}(1,u)
&=-4(m-1)^2
\{
3m(3m-1)u^2+3m(11m-4)u+28m^2-13m+4
\}\\
&\qquad\qquad\qquad\times
(u+2)\{(2m-1)u+3m\}.
\end{align*}

For convenience, set 
\begin{align*}
\mathcal{C}_m
&=2m(3m-1)u^3+3m(11m-4)u^2\\
&\qquad\qquad\qquad+2(28m^2-13m+4)u+6m(5m-4),\\
\mathcal{D}_m
&=(u+2)\{(2m-1)u+3m\}\mathcal{Q}_m,
\end{align*}
where
\begin{align*}
\mathcal{Q}_m
&=3m(3m-1)u^2+3m(11m-4)u+28m^2-13m+4.
\end{align*}

With this notation,
\begin{align*}
P_{0,m}(1,u)
&=-(m-1)\mathcal{C}_m(u),\\
P_{1,m}(1,u)
&=-4(m-1)^2\mathcal{D}_m(u).
\end{align*}

\begin{lemma}\label{l14.3}
For $m\geq3$ with $m\neq4$, the polynomials $\mathcal{C}_m$ and $\mathcal{D}_m$ are relatively prime in $\mathbb{R}[u]$.
\end{lemma}

\begin{proof}
We first confirm that $\mathcal{C}_m(-2)\neq0$ and $\mathcal{C}_m\Big(-\frac{3m}{2m-1}\Big)\neq0$.
By the definition of $\mathcal{C}_m$, we have
\[
\mathcal{C}_m(-2)=2(m-4)(m+2)\neq0
\]
and
\[
\mathcal{C}_m\Big(-\frac{3m}{2m-1}\Big)=-\frac{27m^4}{(2m-1)^3}\neq0.
\]
Hence it remains to show that $\mathcal{C}_m$ and $\mathcal{Q}_m$ are relatively prime in $\mathbb{R}[u]$.

A long but straightforward computation shows that 
\[
\mathcal{C}_m(u)=
\frac{1}{3(3m-1)}
\Big(
\{2(3m-1)u+11m-4\}\mathcal{Q}_m(u)-(\mathcal{A}_mu+\mathcal{B}_m)
\Big),
\]
where 
\begin{align*}
\mathcal{A}_m
&=27m^3+4m^2-52m+16,\\
\mathcal{B}_m
&=38m^3+51m^2+24m-16.
\end{align*}
If $\mathcal{C}_m$ and $\mathcal{Q}_m$ have a nontrivial common factor in $\mathbb{R}[u]$, then the identity above implies that this common factor divides $\mathcal{A}_mu+\mathcal{B}_m$.
Since 
\[
\mathcal{A}_m=m(27m^2-52)+4m^2+16>0\qquad (m\geq3),
\]
the common factor is a constant multiple of $\mathcal{A}_mu+\mathcal{B}_m$.
Thus
\[
\mathcal{Q}_m\Big(-\frac{\mathcal{B}_m}{\mathcal{A}_m}\Big)=0.
\]
However, we have 
\begin{align*}
\mathcal{Q}_m\Big(-\frac{\mathcal{B}_m}{\mathcal{A}_m}\Big)
=-\frac{(3m-1)\mathcal{H}(m)}{\mathcal{A}_m^2}\neq0,
\end{align*}
where 
\begin{align*}
\mathcal{H}(m)
&=150m^7+3724m^6-615m^5-37728m^4\\
&\qquad-24880m^3+35520m^2-10752m+1024.
\end{align*}
In fact, for $m=3$, $\mathcal{H}(3)=-545879\neq0$, and if $m=n+4$, then 
\begin{align*}
\mathcal{H}(n+4)
&=150n^7+7924n^6+139161n^5+1179732n^4\\
&\quad+5383792n^3+13247232n^2+15814656n+6356992>0,
\end{align*}
for $n\geq0$, that is, for $m\geq4$.
Thus $\mathcal{C}_m$ and $\mathcal{D}_m$ are relatively prime in $\mathbb{R}[u].$
\end{proof}

For fixed $m\geq3$ with $m\neq4$, let 
\[
\mathcal{R}_m(y)={\rm Res}_u(P_{0,m}(y,u),P_{1,m}(y,u)).
\]
By Lemmas \ref{lresultant} and \ref{l14.3}, $\mathcal{R}_m(1)\neq0$.
Hence $\mathcal{R}_m$ is not the zero polynomial. Therefore Lemma \ref{lYconst} shows that $y$ is constant on the interval of $\rho$ under consideration.

\subsection{The case $m=4$}

It remains to consider the case $m=4$. In this case, $P_{0,m}$ and $P_{1,m}$ have the common factor $u+2$. In fact, by \eqref{P0m} and \eqref{P1m2}, we have
\begin{align*}
P_{0,4}
&=24(u+2)\mathcal{E}_0(y,u),\\
P_{1,4}
&=-48(u+2)\mathcal{E}_1(y,u),
\end{align*}
where
\begin{align*}
\mathcal{E}_0
&=3u^2y^3-14u^2y^2+2uy^3+4uy^2-44uy\\
&\qquad+8y^2-16y-16,\\
\mathcal{E}_1
&=9u^3y^5-296u^3y^4+868u^3y^3+112u^3y^2\\
&\qquad+20u^2y^5-400u^2y^4-264u^2y^3+4000u^2y^2+352u^2y\\
&\qquad-4uy^5+80uy^4-1992uy^3+3344uy^2+4864uy+128u\\
&\qquad-16y^4-1824y^2+4672y+768.
\end{align*}

In this case, $u+2$ does not vanish on the interval under consideration.
In fact, since $y>1$ and $v>0$,
\[
F_4(y,-2,v)=-(y+2)v\neq0.
\]
Therefore, $P_{0,4}=P_{1,4}=0$ is equivalent to 
\[
\mathcal{E}_0(y,u)=\mathcal{E}_1(y,u)=0.
\]
At $y=1$,
\begin{align}
\mathcal{E}_0(1,u)
&=-11u^2-38u-24,\label{E0}\\
\mathcal{E}_1(1,u)
&=3(7u+12)(33u^2+120u+100)\label{E1}.
\end{align}
We claim that these two polynomials are relatively prime in $\mathbb{R}[u]$.

Substituting $u=-\frac{12}{7}$ into $\mathcal{E}_0$, we have
\[
\mathcal{E}_0\Big(1,-\frac{12}{7}\Big)
=\frac{432}{49}\neq0,
\]
which implies that $7u+12$ is relatively prime to $\mathcal{E}_0(1,u)$.

It remains to consider 
\[
33u^2+120u+100=-3(-11u^2-38u-24)+2(3u+14).
\]
Thus any common factor of $\mathcal{E}_0(1,u)$ and $33u^2+120u+100$ must divide $3u+14$. If such a common factor is nontrivial, it is a constant multiple of $3u+14$. However,
\[
\mathcal{E}_0\Big(1,-\frac{14}{3}\Big)
=-\frac{776}{9}\neq0.
\]
Therefore $\mathcal{E}_0(1,u)$ and $\mathcal{E}_1(1,u)$ are relatively prime in $\mathbb{R}[u]$.

Set 
\[
\mathcal{R}_4(y)={\rm Res}_u(\mathcal{E}_0(y,u),\mathcal{E}_1(y,u)).
\]
By \eqref{E0} and \eqref{E1} and Lemma \ref{lresultant}, $\mathcal{R}_4(1)\neq0$. Thus $\mathcal{R}_4$ is not the zero polynomial. Lemma \ref{lYconst} gives that $y$ is constant on the interval of $\rho$ under consideration. 

\subsection{Final argument}
We have shown that, for $m\geq3$, $y$ is constant on every connected interval of $\rho$ on which the normalized variables are defined.
Hence $y'=0$. By \eqref{y'} and $y>1$, we have $u=-1$. 
By \eqref{vtFm}, $v$ is constant. Hence by \eqref{v'},
\[
0=v'=2v
\Big\{
-\frac{m-y}{m}+4-y
\Big\}.
\]
Thus
\[
y=\frac{3m}{m-1}.
\]
Substituting $u=-1$ and $y=\frac{3m}{m-1}$ into \eqref{Fm}, we have
\[
v=-\frac{3m(m-2)}{m^2+2}<0,
\]
which contradicts $v>0$. Therefore there exists no solution $(y(\rho),u(\rho),v(\rho))$ of 
\eqref{y'}--\eqref{v'} satisfying \eqref{Fm}, $y>1$, and $v>0$.

The preceding argument shows that the mean curvature is constant.

\begin{proposition}\label{pkey}
Let $M^m$ be a connected pseudo-umbilical $\lambda$-biminimal submanifold in $\mathbb{E}^{m+2}$ with $m\geq3$. Then $|{\bf H}|$ is constant.
\end{proposition}

\begin{proof}
Assume that $U$ is nonempty. We can take a point $p\in U$. 
On a sufficiently small neighborhood of $p$, the preceding construction produces functions $y(\rho)$, $u(\rho)$, and $v(\rho)$ satisfying \eqref{y'}--\eqref{v'}, \eqref{Fm}, $y>1$, and $v>0$ along an integral curve of $e_1$. 
The argument above leads to a contradiction. Hence $U$ is empty and $df=0$ on $\{f>0\}$. 

Let $\Omega$ be a connected component of $\{f>0\}$. Then $f$ is constant on $\Omega$, say $c_0(>0)$. If $p\in \overline\Omega$, $f(p)=c_0$, because $f$ is continuous. 
A sufficiently small connected neighborhood of $p$ contained in $\{f>0\}$ meets $\Omega$. 
Therefore it is contained in $\Omega$. Thus $p\in \Omega$, and hence $\Omega$ is also closed. 
Since $M$ is connected, either $f\equiv0$ or $f\equiv c_0(>0)$. 
Therefore $|{\bf H}|$ is constant.  
\end{proof}

We use the following classical result on pseudo-umbilical submanifolds of codimension two with positive constant mean curvature in Euclidean spaces.

\begin{theorem}[\cite{Chen1971a,Chen1971b}]\label{tChen}
Let ${\bf x}:M^m\rightarrow\mathbb{E}^{m+2}$ $(m\geq2)$ be a connected pseudo-umbilical submanifold in the Euclidean space $\mathbb{E}^{m+2}$ with positive constant mean curvature $f=|{\bf H}|$. Then there exist $\tilde c\in\mathbb{E}^{m+2}$ and $R>0$ such that 
\[
{\bf x}(M)\subset\mathbb{S}^{m+1}(\tilde c,R),
\]
and $M$ is minimal in this hypersphere.
Moreover, 
\[
R=\frac{1}{f},\qquad \nabla^\perp{\bf H}=0.
\]
\end{theorem} 

\begin{remark}
In Chen's original paper \cite{Chen1971a}, a pseudo-umbilical submanifold is defined under the assumption that the mean curvature vector is nowhere zero. 
\end{remark}

Finally, we show Theorem \ref{main}.

\begin{proof}[Proof of Theorem \ref{main}]
By Proposition \ref{pkey}, $f$ is constant. 
If $f\equiv0$, then $M$ is minimal. We consider the non-minimal case. 

By Theorem \ref{tChen}, we have $\Delta^\perp {\bf H}=0$. Substituting this into \eqref{plb}, we have
\[
f^2=-\frac{\lambda}{m}.
\]
If $\lambda\geq0$, we have a contradiction. Hence $\lambda<0$. 
Therefore, 
\[
R=\frac{1}{f}=\sqrt{-\frac{m}{\lambda}},
\]
and $M$ is minimal in the hypersphere 
\[
\mathbb{S}^{m+1}\left(\tilde c, \sqrt{-\frac{m}{\lambda}}\right)\subset \mathbb{E}^{m+2}.
\]
This completes the proof of Theorem \ref{main}.
\end{proof}

\section{Proof of Theorem \ref{sub}}

In this section, we prove Theorem \ref{sub}. We proved the result for $c=0$ in the previous section. In the following, assume that $c\neq0$. 
In this case, since $\lambda=0$, we obtain $\beta=1$.
First, we show that the set $U$ is empty. We divide the proof into two cases.

\subsection{The case $m\neq4$}
We evaluate $P_{0,m}(y,u)$ at $y=1$. By \eqref{P0m},
\begin{align*}
P_{0,m}(1,u)
&=-(m-1)
\{
2m(3m-1)u^3+2m(17m-6)u^2\\
&\qquad\qquad\qquad+2(29m^2-13m+4)u+6m(5m-4)
\}.
\end{align*}
By \eqref{P1m2},
\begin{align*}
P_{1,m}(1,u)
&=-4(m-1)^2
\{
3m(3m-1)u^2+2m(17m-6)u+29m^2-13m+4
\}\\
&\qquad\qquad\qquad\times
(u+2)\{(2m-1)u+3m\}.
\end{align*}

For convenience, set 
\begin{align*}
\mathcal{C}_m
&=2m(3m-1)u^3+2m(17m-6)u^2\\
&\qquad\qquad\qquad+2(29m^2-13m+4)u+6m(5m-4),\\
\mathcal{D}_m
&=(u+2)\{(2m-1)u+3m\}\mathcal{Q}_m,
\end{align*}
where
\begin{align*}
\mathcal{Q}_m
&=3m(3m-1)u^2+2m(17m-6)u+29m^2-13m+4.
\end{align*}

With this notation,
\begin{align*}
P_{0,m}(1,u)
&=-(m-1)\mathcal{C}_m(u),\\
P_{1,m}(1,u)
&=-4(m-1)^2\mathcal{D}_m(u).
\end{align*}

\begin{lemma}\label{subl14.3}
For $m\geq3$ with $m\neq4$, the polynomials $\mathcal{C}_m$ and $\mathcal{D}_m$ are relatively prime in $\mathbb{R}[u]$.
\end{lemma}

\begin{proof}
We first confirm that $\mathcal{C}_m(-2)\neq0$ and $\mathcal{C}_m\Big(-\frac{3m}{2m-1}\Big)\neq0$.
By the definition of $\mathcal{C}_m$, we have
\[
\mathcal{C}_m(-2)=2(m-4)(m+2)\neq0
\]
and
\[
\mathcal{C}_m\Big(-\frac{3m}{2m-1}\Big)=-\frac{6m^3(m+1)^2}{(2m-1)^3}\neq0.
\]
Hence it remains to show that $\mathcal{C}_m$ and $\mathcal{Q}_m$ are relatively prime in $\mathbb{R}[u]$.

A long but straightforward computation shows that 
\[
\mathcal{C}_m(u)=
\frac{1}{9(3m-1)}
\Big(
\{6(3m-1)u+2(17m-6)\}\mathcal{Q}_m(u)-(\mathcal{A}_mu+\mathcal{B}_m)
\Big),
\]
where 
\begin{align*}
\mathcal{A}_m
&=4(28m^3-39m+12),\\
\mathcal{B}_m
&=4(44m^3+32m^2+19m-12).
\end{align*}
If $\mathcal{C}_m$ and $\mathcal{Q}_m$ have a nontrivial common factor in $\mathbb{R}[u]$, then the identity above implies that this common factor divides $\mathcal{A}_mu+\mathcal{B}_m$.
Since 
\[
\mathcal{A}_m=4m(28m^2-39)+48>0\qquad (m\geq3),
\]
the common factor is a constant multiple of $\mathcal{A}_mu+\mathcal{B}_m$.
Thus
\[
\mathcal{Q}_m\Big(-\frac{\mathcal{B}_m}{\mathcal{A}_m}\Big)=0.
\]
However, we have 
\begin{align*}
\mathcal{Q}_m\Big(-\frac{\mathcal{B}_m}{\mathcal{A}_m}\Big)
=-\frac{(3m-1)\mathcal{H}(m)}{\mathcal{A}_m^2}\neq0,
\end{align*}
where 
\begin{align*}
\mathcal{H}(m)
&=576\Big\{
16m^7+64m^6-40m^5-596m^4\\
&\qquad\qquad-402m^3+561m^2-168m+16
\Big\}.
\end{align*}
In fact, if $m=n+3$, then 
\begin{align*}
\mathcal{H}(n+3)
&=576\Big\{16n^7+400n^6+4136n^5+22564n^4\\
&\qquad\qquad+68766n^3+113367n^2+86736n+17359
\Big\}>0,
\end{align*}
for $n\geq0$, that is, for $m\geq3$.
Thus $\mathcal{C}_m$ and $\mathcal{D}_m$ are relatively prime in $\mathbb{R}[u].$
\end{proof}

For fixed $m\geq3$ with $m\neq4$, let 
\[
\mathcal{R}_m(y)={\rm Res}_u(P_{0,m}(y,u),P_{1,m}(y,u)).
\]
By Lemmas \ref{lresultant} and \ref{subl14.3}, $\mathcal{R}_m(1)\neq0$.
Hence $\mathcal{R}_m$ is not the zero polynomial. Therefore Lemma \ref{lYconst} shows that $y$ is constant on the interval of $\rho$ under consideration.

\subsection{The case $m=4$}

It remains to consider the case $m=4$. In this case, $P_{0,m}$ and $P_{1,m}$ have the common factor $u+2$. In fact, by \eqref{P0m} and \eqref{P1m2}, we have
\begin{align*}
P_{0,4}
&=8(u+2)\mathcal{E}_0(y,u),\\
P_{1,4}
&=-16(u+2)\mathcal{E}_1(y,u),
\end{align*}
where
\begin{align*}
\mathcal{E}_0
&=y^2(16y-49)u^2+2y(8y^2+13y-81)u+8(7y^2-9y-7),\\
\mathcal{E}_1
&=3y^2(128y^3-856y^2+1225y+196)u^3\\
&\quad+12y(64y^4-216y^3-565y^2+1503y+162)u^2\\
&\quad+4(64y^5+608y^4-4239y^3+2477y^2+5908y+168)u\\
&\quad+16(56y^4-23y^3-987y^2+1432y+224).
\end{align*}

In this case, $u+2$ does not vanish on the interval under consideration.
In fact, since $y>1$ and $v>0$,
\[
F_4(y,-2,v)=-3v\neq0.
\]
Therefore, $P_{0,4}=P_{1,4}=0$ is equivalent to 
\[
\mathcal{E}_0(y,u)=\mathcal{E}_1(y,u)=0.
\]
At $y=1$,
\begin{align}
\mathcal{E}_0(1,u)
&=-3(11u^2+40u+24),\label{subE0}\\
\mathcal{E}_1(1,u)
&=9(7u+12)(33u^2+124u+104)\label{subE1}.
\end{align}
We claim that these two polynomials are relatively prime in $\mathbb{R}[u]$.

Substituting $u=-\frac{12}{7}$ into $\mathcal{E}_0$, we have
\[
\mathcal{E}_0\Big(1,-\frac{12}{7}\Big)
=\frac{1800}{49}\neq0,
\]
which implies that $7u+12$ is relatively prime to $\mathcal{E}_0(1,u)$.

It remains to consider 
\[
33u^2+124u+104=-3(-11u^2-40u-24)+4(u+8).
\]
Thus any common factor of $\mathcal{E}_0(1,u)$ and $33u^2+124u+104$ must divide $u+8$. If such a common factor is nontrivial, it is a constant multiple of $u+8$. However,
\[
\mathcal{E}_0\Big(1,-8\Big)
=-1224\neq0.
\]
Therefore $\mathcal{E}_0(1,u)$ and $\mathcal{E}_1(1,u)$ are relatively prime in $\mathbb{R}[u]$.

Set 
\[
\mathcal{R}_4(y)={\rm Res}_u(\mathcal{E}_0(y,u),\mathcal{E}_1(y,u)).
\]
By \eqref{subE0} and \eqref{subE1} and Lemma \ref{lresultant}, $\mathcal{R}_4(1)\neq0$. Thus $\mathcal{R}_4$ is not the zero polynomial. Lemma \ref{lYconst} gives that $y$ is constant on the interval of $\rho$ under consideration. 

\subsection{Final argument}
We have shown that, for $m\geq3$, $y$ is constant on every connected interval of $\rho$ on which the normalized variables are defined.
Hence $y'=0$. By \eqref{y'} and $y>1$, we have $u=-1$. 
By \eqref{vtFm}, $v$ is constant. Hence by \eqref{v'},
\[
0=v'=2v
\Big\{
-\frac{m-y}{m}+4-y
\Big\}.
\]
Thus
\[
y=\frac{3m}{m-1}.
\]
Substituting $u=-1$ and $y=\frac{3m}{m-1}$ into \eqref{Fm}, we have
\[
v=-\frac{m(2m-5)}{(m-1)^2}<0,
\]
which contradicts $v>0$. Therefore there exists no solution $(y(\rho),u(\rho),v(\rho))$ of 
\eqref{y'}--\eqref{v'} satisfying \eqref{Fm}, $y>1$, and $v>0$.

The preceding argument shows that the mean curvature is constant.

\begin{proposition}\label{subpkey}
Let $M^m$ be a connected pseudo-umbilical $0$-biminimal submanifold in $N^{m+2}(c)$ with $m\geq3$. Then $|{\bf H}|$ is constant.
\end{proposition}

\begin{proof}
Assume that $U$ is nonempty. We can take a point $p\in U$. 
On a sufficiently small neighborhood of $p$, the preceding construction produces functions $y(\rho)$, $u(\rho)$, and $v(\rho)$ satisfying \eqref{y'}--\eqref{v'}, \eqref{Fm}, $y>1$, and $v>0$ along an integral curve of $e_1$. 
The argument above leads to a contradiction. Hence $U$ is empty and $df=0$ on $\{f>0\}$. 

Let $\Omega$ be a connected component of $\{f>0\}$. Then $f$ is constant on $\Omega$, say $c_0(>0)$. If $p\in \overline\Omega$, $f(p)=c_0$, because $f$ is continuous. 
A sufficiently small connected neighborhood of $p$ contained in $\{f>0\}$ meets $\Omega$. 
Therefore it is contained in $\Omega$. Thus $p\in \Omega$, and hence $\Omega$ is also closed. 
Since $M$ is connected, either $f\equiv0$ or $f\equiv c_0(>0)$. 
Therefore $|{\bf H}|$ is constant.  
\end{proof}

Once we obtain that $|{\bf H}|$ is constant, we obtain Theorem \ref{sub}.
When $c<0$, Proposition~2.2 in \cite{Oniciuc2002} shows that $M$ is minimal. 
When $c>0$, if $M$ is not minimal, then 
$\nabla^\perp{\bf H}=0$ and $|{\bf H}|^2=c>0$. 
In the unit sphere, every such non-minimal submanifold is minimal in a small hypersphere $\mathbb{S}^{m+1}\left(\frac{1}{\sqrt{2}}\right)\subset \mathbb{S}^{m+2}$ (cf. \cite{Oniciuc2003}, \cite{BMO2008}).

This completes the proof of Theorem \ref{sub}.

\quad\\

\subsection*{Acknowledgements}~\\
The author would like to express his deep gratitude to Cezar Oniciuc for his useful comments on the first draft.


\begin{thebibliography}{99} 

\bibitem{AM2013} K. Akutagawa and S. Maeta, 
\textit{Biharmonic properly immersed submanifolds in Euclidean spaces}, 
Geom. Dedicata {\bf 164} (2013), 351--355. 

\bibitem{BMO2008} A. Balmu\c{s}, S. Montaldo, and C. Oniciuc,
\textit{Classification results for biharmonic submanifolds in spheres},
Israel J. Math. {\bf 168} (2008), 201--220.

\bibitem{BMO2013} A. Balmu\c{s}, S. Montaldo, and C. Oniciuc, 
\textit{Biharmonic PNMC submanifolds in spheres}, 
Ark. Mat. {\bf 51} (2013), 197--221. 

\bibitem{Brinkmann1925} H. W. Brinkmann,
\textit{Einstein spaces which are mapped conformally on each other}, 
Math. Ann. \textbf{94} (1925), 119--145.

\bibitem{CMO2002} R. Caddeo, S. Montaldo, and C. Oniciuc,
\textit{Biharmonic submanifolds in spheres},
Israel J. Math. {\bf 130} (2002), 109--123.

\bibitem{Chen1971a} B.-Y. Chen, 
\textit{Minimal hypersurfaces in an $m$-sphere}, 
Proc. Amer. Math. Soc. {\bf 29} (1971), 375--380. 

\bibitem{Chen1971b} B.-Y. Chen, 
\textit{On the mean curvature of submanifolds of Euclidean space}, 
Bull. Amer. Math. Soc. {\bf 77} (1971), 741--743. 

\bibitem{Chen1976} B.-Y. Chen, 
\textit{Total mean curvature of immersed surfaces in $E^m$}, 
Trans. Amer. Math. Soc. {\bf 218} (1976), 333--341. 

\bibitem{Chen1988} B.-Y. Chen, 
\textit{Some open problems and conjectures on submanifolds of finite type},
Michigan State University, (1988~version).


\bibitem{Chen2019} B.-Y. Chen, 
\textit{Chen's biharmonic conjecture and submanifolds with parallel normalized mean curvature vector}, 
Mathematics {\bf 7} (2019), no.~8, Article 710. 

\bibitem{Chen2025} B.-Y. Chen, 
\textit{Recent developments in Chen's biharmonic conjecture and some related topics}, 
Mathematics {\bf 13} (2025), no.~9, Article 1417. 

\bibitem {Chen-Ishikawa-1} B.-Y.~Chen and S.~Ishikawa, 
    {\it Biharmonic surfaces in pseudo-Euclidean spaces}, 
    Memoirs Fac.\ Sci., Kyushu Univ., Ser.~A {\bf 45} (1991), 323--347. 

\bibitem {Chen-Ishikawa-2} B.-Y.~Chen and S.~Ishikawa, 
    {\it Biharmonic pseudo-Riemannian submanifolds in pseudo-Euclidean spaces}, 
    Kyushu J.\ Math.\ {\bf 52} (1998), 167--185. 

\bibitem{CL1972} B.-Y. Chen and G.~D. Ludden, 
\textit{Rigidity theorems for surfaces in Euclidean space}, 
Bull. Amer. Math. Soc. {\bf 78} (1972), 72--73. 

\bibitem{CLO2005} D.~A. Cox, J. Little, and D. O'Shea, 
\textit{Using Algebraic Geometry}, 2nd ed., Graduate Texts in Mathematics, Vol.~185, Springer, New York, 2005. 

\bibitem{Defever1998} F. Defever, 
\textit{Hypersurfaces of ${\bf E}^4$ with harmonic mean curvature vector}, 
Math. Nachr. {\bf 196} (1998), 61--69. 

\bibitem{Dimitric1989} I. Dimitri\'c, 
\textit{Quadric representation and submanifolds of finite type}, 
Ph.D. Thesis, Michigan State University, 1989. 


\bibitem{EL1983} J. Eells and L. Lemaire, 
\textit{Selected topics in harmonic maps}, 
CBMS Regional Conference Series in Mathematics, No. 50, Amer. Math. Soc., Providence, RI, 1983. 

\bibitem{Fetcu2012} D. Fetcu, 
\textit{Surfaces with parallel mean curvature vector in complex space forms}, 
J. Differential Geom. {\bf 91} (2012), no.~2, 215--232. 

\bibitem{FHZ2021} Y. Fu, M.~C. Hong, and X. Zhan, 
\textit{On Chen's biharmonic conjecture for hypersurfaces in $\mathbb{R}^5$}, 
Adv. Math. {\bf 383} (2021), Paper No. 107697, 28 pp. 

\bibitem{FHZ2023} Y. Fu, M.~C. Hong, and X. Zhan, 
\textit{Biharmonic conjectures on hypersurfaces in a space form}, 
Trans. Amer. Math. Soc. {\bf 376} (2023), no.~12, 8411--8445. 

\bibitem{HV1995} T. Hasanis and T. Vlachos, 
\textit{Hypersurfaces in $E^4$ with harmonic mean curvature vector field}, 
Math. Nachr. {\bf 172} (1995), 145--169. 

\bibitem{IL2012} J. Inoguchi and J.-E. Lee,
\textit{Biminimal curves in 2-dimensional space forms},
Commun. Korean Math. Soc. {\bf 27} (2012), no.~4, 771--780.

\bibitem{Jiang1986} G. Y. Jiang, \textit{2-Harmonic maps and their first and second variational formulas}, 
Chin. Ann. Math. Ser. A {\bf 7} (1986), 389--402. 

\bibitem{Jiang1987} G. Y. Jiang, \textit{Some non-existence theorems of 2-harmonic isometric immersions into Euclidean spaces}, 
Chin. Ann. Math. Ser. A {\bf 8} (1987), 376--383. 

\bibitem{LM2008} E. Loubeau and S. Montaldo, 
\textit{Biminimal immersions}, 
Proc. Edinburgh Math. Soc. {\bf 51} (2008), 421--437. 

\bibitem{LO2016} E. Loubeau and C. Oniciuc,
\textit{Constant mean curvature proper-biharmonic surfaces of constant Gaussian curvature in spheres},
J. Math. Soc. Japan {\bf 68} (2016), no.~3, 997--1024.

\bibitem{Maeta2012} S. Maeta, 
\textit{Biminimal properly immersed submanifolds in the Euclidean spaces}, 
J. Geom. Phys. {\bf 62} (2012), 2288--2293. 

\bibitem{Maeta2014} S. Maeta, 
\textit{Properly immersed submanifolds in complete Riemannian manifolds}, 
Adv. Math. {\bf 253} (2014), 139--151. 

\bibitem{Maeta2021} S. Maeta,
\textit{Classification of generalized Yamabe solitons under vanishing conditions on the Weyl, Cotton, and Cao-Chen tensors},
arXiv:2107.05487 [math. DG].

\bibitem{Maeta2026} S. Maeta, 
\textit{Biharmonic rotational surfaces in the four-dimensional Euclidean space are minimal}, 
arXiv:2605.09587. 

\bibitem{MU2013} S. Maeta and H. Urakawa, 
\textit{Biharmonic Lagrangian submanifolds in K\"ahler manifolds}, 
Glasgow Math. J. {\bf 55} (2013), no.~2, 465--480. 

\bibitem{ONeill1983}
B. O'Neill,
\textit{Semi-Riemannian Geometry With Applications to Relativity},
Academic Press, New York, 1983.

\bibitem{Oniciuc2002}
C. Oniciuc,
\textit{Biharmonic maps between Riemannian manifolds},
An. Stiint. Univ. Al. I. Cuza Iasi Mat. (N.S.) {\bf 48} (2002), 237--248.

\bibitem{Oniciuc2003} C. Oniciuc, 
\textit{Tangency and Harmonicity Properties}, 
Ph.D. Thesis, Geometry Balkan Press, Bucharest, 2003.

\bibitem{Petersen2016}
P. Petersen,
{\it Riemannian Geometry}, 3rd ed.,
Graduate Texts in Mathematics, {\bf 171},  Springer, 2016.

\bibitem{Sasahara2010} T. Sasahara,
\textit{A classification result for biminimal Lagrangian surfaces
in complex space forms},
J. Geom. Phys. {\bf 60} (2010), no.~6--8, 884--895.

\bibitem{Smyth1973} B. Smyth, 
\textit{Submanifolds of constant mean curvature}, 
Math. Ann. {\bf 205} (1973), 265--280. 

\bibitem{ST2018} R. Ye\u{g}in \c{S}en and N.~C. Turgay, 
\textit{On biconservative surfaces in 4-dimensional Euclidean space}, 
J. Math. Anal. Appl. {\bf 460} (2018), no.~2, 565--581. 

\bibitem{Tashiro1965}Y. Tashiro, 
\textit{Complete Riemannian manifolds and some vector fields},
 Trans. Amer. Math. Soc. {\bf 117} (1965), 251--275.
 
\bibitem{YI1971} K. Yano and S. Ishihara, 
\textit{Submanifolds with parallel mean curvature vector}, 
J. Differential Geom. {\bf 6} (1971), 95--118. 

\bibitem{Yau1974} S.-T. Yau, 
\textit{Submanifolds with constant mean curvature}, 
Amer. J. Math. {\bf 96} (1974), 346--366.


     
     
     
\end{thebibliography}


\bibliographystyle{amsbook}

\end{document}